\documentclass[a4paper, 11pt, reqno]{amsart}
\usepackage{amsmath,amsfonts,amssymb,amsthm,enumerate}
\usepackage{graphicx}
\usepackage{xy} \xyoption{all}

\newtheorem{thm}{Theorem}[section]
\newtheorem{cor}[thm]{Corollary}
\newtheorem{prop}[thm]{Proposition}
\newtheorem{lem}[thm]{Lemma}

\theoremstyle{definition}
\newtheorem{defn}[thm]{Definition}
\newtheorem{exas}[thm]{Example}

\let\phi\varphi

\begin{document}
	\title{Leavitt path algebras having graded  Invariant Basis Number}
	\maketitle
	
	\begin{center}
		Ngo Tan Phuc\footnote{Faculty of Mathematics - Informatics Teacher Education, Dong Thap University, Vietnam. E-mail address: \texttt{ntphuc@dthu.edu.vn}  
			
			\ \ {\bf Acknowledgements}:   The author expresses his deep gratitude to Professor Tran Giang Nam (Institute of Mathematics, VAST,  Vietnam)  and Professor Roozbeh Hazrat (Centre for Research in Mathematics and Data Science, Western Sydney University, Australia) for their valuable suggestions
			which led to the final shape of the paper. The author is partially supported by the Vietnam Ministry of Education and Training under the grant
			number B2026-CTT-02. 
		}
		
	\end{center}
	
	\begin{abstract} 
	 In this paper, we study the Graded Invariant Basis Number (gr-IBN) property for Leavitt path algebras of finite graphs. Using the talented monoid as our main tool, we establish a complete matrix-theoretic characterization of when a Leavitt path algebra of a finite graph fails to have gr-IBN. Consequently, we identify several classes of graphs whose Leavitt path algebras have gr-IBN, including graphs with sinks, Cayley graphs, and Hopf graphs associated with finite groups. We also investigate the preservation of gr-IBN under quotients by hereditary saturated subsets and under Cartesian products of graphs.
			
		\medskip
		
		
	\end{abstract}
	
	\section{Introduction}
	Given a (row-finite) directed graph $E$ and a field $K$, Abrams and Aranda Pino
	\cite{ap:tlpaoag05}, and, independently Ara, Moreno, and Pardo \cite{amp:nktfga},
	introduced the \emph{Leavitt path algebra} $L_K(E)$. Abrams and Aranda Pino \cite{ap:tlpaoag08} later
	extended this definition  to all countable directed graphs. Goodearl \cite{g:lpaadl09} extended the notion of Leavitt path algebras to all (possibly uncountable) directed graphs. In \cite{tomf:lpawciacr}, Tomforde generalized the construction of Leavitt path algebras by replacing the field with a commutative ring. Katsov, Nam and Zumbr\"{a}gel \cite{knz:solpawcicr} considered the concept of Leavitt path algebras with coefficients in a commutative semiring. Leavitt path algebras generalize the Leavitt algebras $L_K(1, n)$ of \cite{leav:tmtoar}, and also contain many other interesting classes of algebras. In addition, Leavitt path algebras are intimately related to graph $C^*$-algebras (see~\cite{r:ga}). During the past twenty years, Leavitt path algebras have become a topic of intense investigation across multiple areas of mathematics. For a detailed history and overview of Leavitt path algebras, we refer the reader to the survey article~\cite{a:firstde}.
	
	{\it Invariant Basis Number} (IBN) is a fundamental ring-theoretic property: a unital ring $R$ has IBN if the isomorphism $R^m \cong R^n$ (as right $R$-modules) implies $m = n$. The IBN property of Leavitt path algebra has been studied in several papers \cite{ak:cpahibn, kbgm:colpawtv, ko:clpaatibnp, np:tsolpaatibnp}. In particular, Kanuni and \"{O}zaydin \cite{ko:clpaatibnp}  and Nam and the author \cite{np:tsolpaatibnp} gave criteria for Leavitt path algebras of finite graphs to have~IBN.
	
	{\it Graded Invariant Basis Number} (gr-IBN) is a refinement of IBN for graded rings: a $\mathbb{Z}$-graded ring $R$ has gr-IBN if the isomorphism $R(\alpha_1)\oplus  \cdots \oplus R(\alpha_m) \cong_{gr} R(\beta_1)\oplus  \cdots \oplus R(\beta_n)$ (as graded modules) implies $m = n$. Since graded isomorphisms are more restrictive than ungraded ones, IBN implies gr-IBN, but the converse is not obvious.

 	Hazrat and Li \cite{hl:tmLpa} established a deep connection between the geometry of a graph $E$ and the algebraic structure of its \emph{talented monoid} $T_E$, the graded analogue of the well-known graph monoid. This parallels the correspondence between the geometry of $E$ and the algebraic structure of the Leavitt path algebra $L_K(E)$. Crucially, $T_E$ can be realized as the graph monoid of the covering graph $\overline{E}$, which allows us to apply the Confluence Lemma and translate questions about graded module isomorphisms into equalities in $T_E$.
 	
 	Motivated by this machinery, we extend the IBN problem to the graded setting. Since Leavitt path algebras are $\mathbb{Z}$-graded algebras, it is natural to investigate their gr-IBN property. The main goal of this paper is to characterize when a Leavitt path algebra has gr-IBN, and to identify several classes of graphs for which this property holds.

	The paper is organized as follows. Section~2 reviews the necessary background on Leavitt path algebras, graded modules, IBN and gr-IBN, the talented monoid and its connection to the covering graph. We identify several classes of graphs whose Leavitt path algebras have gr-IBN (Corollaries \ref{maximal} and \ref{atmost}).
	In Section~3, we prove that the Leavitt path algebras of graphs with sinks have gr-IBN (Proposition \ref{hassink}) and   provide a complete criterion for $L_K(E)$ to fail gr-IBN in terms of the adjacency matrix $A_E$ (Theorem \ref{thm-nogribn}). In Section 4, based on the results of Section 3, we characterize the gr-IBN property for Leavitt path algebras of Cayley graphs and Hopf graphs associated with finite groups (Propositions \ref{Caley} and  \ref{HopfThm}). 
	We also show that the gr-IBN property of $L_K(E)$ can be reduced to that of $L_K(E/H)$, where $H$ is a hereditary saturated subset of $E^0$ (Proposition \ref{cor4.7}). Finally, we apply this result to establish the gr-IBN property for Leavitt path algebras of Cartesian products of graphs (Proposition \ref{source}), which is an extension of \cite[Proposition 3.15]{Li2024}.

	\section{Preliminaries}
	In this section, we fix notation and recall standard definitions regarding Leavitt path algebras, graded modules, graph monoids, and the talented monoid—the main technical tool of this paper. 
	
	\subsection{Graphs and Leavitt path algebras}
	
	We begin by recalling some general notions of graph theory: A (directed) graph
	$E = (E^0, E^1, s, r)$ (or shortly $E = (E^0, E^1)$)
	consists of two disjoint sets $E^0$ and $E^1$, called \emph{vertices} and \emph{edges}
	respectively, together with two maps $s, r: E^1 \longrightarrow E^0$.  The
	vertices $s(e)$ and $r(e)$ are referred to as the \emph{source} and the \emph{range}
	of the edge~$e$, respectively. The graph is called \emph{row-finite} if
	$|s^{-1}(v)|< \infty$ for all $v\in E^0$. All graphs in this paper will be assumed
	to be row-finite. A graph $E$ is \emph{finite} if both sets $E^0$ and $E^1$ are finite
	(or equivalently, when $E^0$ is finite, by the row-finite hypothesis).
	A vertex~$v$ for which $s^{-1}(v)$ is empty is called a \emph{sink}; a vertex~$v$ for which
	$r^{-1}(v)$ is empty is called a \emph{source}; a vertex~$v$ is called an \emph{isolated vertex}
	if it is both a source and a sink; and a vertex~$v$ is \emph{regular}
	iff $0 < |s^{-1}(v)| < \infty$. 
	%A graph $E$ is said to be \emph{source-free} if it has
	%no sources.
	
	A \emph{path} $p = e_{1} \dots e_{n}$ in a graph $E$ is a sequence of
	edges $e_{1}, \dots, e_{n}$ such that $r(e_{i}) = s(e_{i+1})$ for all $i
	= 1, \dots, n-1$.  In this case, we say that the path~$p$ starts at
	the vertex $s(p) := s(e_{1})$ and ends at the vertex $r(p) :=
	r(e_{n})$, and has \emph{length} $|p| := n$. If there exists a path $p$ with $s(p) =v, r(p)=w$, then we write $v\geq w$. We denote by $p^0$
	the set of its vertices, that is, $p^0 = \{s(e_i), r(e_i)\ |\ i = 1,...,n\}$.
	%We consider the vertices in~$E^0$ to be paths of length~$0$.
	If $p$ is a path in $E$, and if $v= s(p) = r(p)$, then~$p$ is a \emph{closed path
		based at} $v$. A closed path based at~$v$, $p = e_{1} \dots e_{n}$, is
	a \emph{closed simple path based at}~$v$ if $s(e_i) \neq v$ for
	every $i > 1$. If $p = e_{1} \dots e_{n}$ is a closed path and all vertices
	$s(e_{1}), \dots, s(e_{n})$ are distinct, then the subgraph
	$(s(e_{1}), \dots, s(e_{n}); e_{1}, \dots, e_{n})$ of the graph $E$
	is called a \emph{cycle}.  An edge~$f$ is an \emph{exit} for a path
	$p = e_{1} \dots e_{n}$ if $s(f) = s(e_{i})$ but $f \ne e_{i}$ for
	some $1 \le i \le n$. 
%	A graph $E$ is \emph{acyclic} if it has no cycles;
%	and the graph $E$ is said to be a \emph{no-exit graph} if no cycle in $E$ has an exit.
	
%	For any graph $E= (E^0, E^1)$ and any $v\in E^0$, the set $T(v):= \{w\in E^0\ |\
%	v\geq w\}$ is the \emph{tree} of $v$, that is, the set of all the vertices in the graph
%	$E$ which follow $v$ ($v\geq w$ means that there exists a path $p$ with $s(p) =v$
%	and $r(p) = w$). 
	%We will denote it by $T_E(v)$ when it is necessary to emphasize the dependence on the graph $E$.

	A subset $H \subseteq E^0$ is said to be {\it hereditary} if $r(e) \in H$ for any $v \in H$ and $e \in s^{-1}(v)$.
	We say that $H$ is {\it saturated} if, for every regular vertex $v \in E^0$, the inclusion
	$
	r\big(s^{-1}(v)\big) \subseteq H
	$
	implies that $v \in H$.
	
	%Denote by $\mathcal{H}_E$ the set of all hereditary saturated subsets of $E^0$.
	%The sets $\varnothing$ and $E^0$ are called the trivial hereditary saturated subsets of $E^0$.

	For any graph $E= (E^0, E^1)$, we denote by $A_E$ the \emph{adjacency matrix} of
	$E$. Formally, if $E^0 = \{v_1,..., v_n\}$, then $A_E = (a_{ij})$ the $n\times n$
	matrix for which $a_{ij}$ is the number of edges having $s(e)=v_i$ and
	$r(e)=v_j$. %Specially, if $v_i\in E^0$ is a sink, then $a_{ij} = 0$ for all
	%$j=1,...,n$.

%	To a directed graph, one can associate an algebra generated by vertices and edges, subject to relations that "locally" on each vertex resemble those that were considered by William Leavitt in his seminal papers in the 1960's (see [2] for a comprehensive history). Such algebras, when associated to non-cyclic strongly connected graphs, are purely infinite simple, that is, each one-sided ideal contains an infinite idempotent.
	
	\begin{defn} [{\cite[Definition 1.2.3]{aam:lpa}}]
	Let $E$ be an arbitrary graph and $K$ any field, the \emph{Leavitt path algebra}
	$L_K(E)$ of the graph $E$ with coefficients in $K$ is the $K$-algebra generated by the sets
	$E^0$ and $E^1$ together with a set of variables $\{e^{*} \mid e \in E^1\}$, satisfying the following relations
	for all $v, w \in E^0$ and $e, f \in E^1$:
	\begin{enumerate}
		\item[(1)] $vw = \delta_{v,w}\, w$;
		
		\item[(2)] $s(e)e = e = er(e)$ and $r(e)e^{*} = e^{*} = e^{*}s(e)$;
		
		\item[(3)] $e^{*}f = \delta_{e,f}\, r(e)$;
		
		\item[(4)] $v = \sum_{e \in s^{-1}(v)} ee^{*}$ for any regular vertex $v$.
	\end{enumerate}
	\end{defn}
	
	{\it We mention that throughout the note, $K$ is an arbitrary field.} 
	
	If the graph $E$ is finite, then $L_K(E)$ is a unital ring having identity
	$
	1 = \sum_{v \in E^0} v
	$
	(see, e.g., \cite[Lemma 1.6]{ap:tlpaoag05}). For any path $p = e_1 e_2 \cdots e_n$, the element
	$e_n^{*} \cdots e_2^{*} e_1^{*}$ of $L_K(E)$ is denoted by $p^{*}$.
	It can be shown (\cite[Lemma 1.7]{ap:tlpaoag05}) that $L_K(E)$ is spanned as a $K$-vector space by
	$$
	\{\, p q^{*} \mid p,q \in E^{*},\ r(p)=r(q) \,\}.
	$$
	Indeed, $L_K(E)$ is a $\mathbb{Z}$-graded $K$-algebra
	$$
	L_K(E) = \bigoplus_{n \in \mathbb{Z}} L_K(E)_n,
	$$
	where for each $n \in \mathbb{Z}$, the degree-$n$ component $L_K(E)_n$ is the set
	$$
	\operatorname{span}\{\, p q^{*} \mid p,q \in \mathrm{Path}(E),\ r(p)=r(q),\ |p|-|q|=n \,\}.
	$$
	
	\subsection{Ring-theoretic concepts}
	
	Let $\Gamma$ be an abelian group. A ring $R$ (possibly without unit) is called a $\Gamma$-\emph{graded ring} if $R=\oplus_{\gamma\in \Gamma}R_{\gamma}$ such that each $R_{\gamma}$ is an additive subgroup of $R$ and $R_{\lambda}R_{\gamma}\subseteq R_{\lambda+\gamma}$ for all $\lambda, \gamma\in \Gamma$. Let $R$ be a $\Gamma$-graded ring. A \emph{graded right $R$-module} $M$ is defined to be a right $R$-module $M$ with a direct sum decomposition $M=\oplus_{\gamma\in \Gamma}M_{\gamma}$, where each $M_{\gamma}$ is an additive subgroup of $M$ such that $M_{\lambda}R_{\gamma}\subseteq M_{\lambda+\gamma}$ for all $\lambda, \gamma \in \Gamma$. For $\Gamma$-graded right $R$-modules $M$ and $N$, a  $\Gamma$-\emph{graded module homomorphism} $f: M\longrightarrow N$ is a module homomorphism such that $f(M_{\gamma})\subseteq N_{\gamma}$ for all $\gamma \in \Gamma$. A graded homomorphism $f$ is called \emph{graded module isomorphism} if it is bijective and we write $M\cong_{gr} N.$ 
	
	Let $M$ be a $\Gamma$-graded right $R$-modules. For $\delta\in \Gamma$, we define the $\delta$-\emph{shifted} graded right $R$-module $M(\delta)$ as $M(\delta)=\oplus_{\gamma\in \Gamma}M(\delta)_{\gamma}$, where $M(\delta)_{\gamma}=M_{\delta + \gamma}.$
	
	For a ring $R$ with unit, the isomorphism classes of finitely generated projective
	(left/right) $R$-modules with direct sum as the operation form a monoid denoted by $\mathcal{V}(R)$.
	This construction can be extended to non-unital rings via idempotents. 
	For a $\Gamma$-graded ring $R$, considering the graded finitely generated projective modules, it provides us with the monoid $\mathcal{V}^{\text{gr}}(R)$  which has an action of $\Gamma$ on it via the shift operation on modules.
	In this article we consider these monoids when the algebra is a Leavitt path algebra associated to a row-finite graph $E$.
	
	\begin{defn} (1) (\cite[page 3]{c:firaligr})
		A ring $R$ is said to have \emph{Invariant Basis Number} (for short, IBN) if,
		for any pair of positive integers $m \text{ and } n$, $$R^m\cong R^n \text{ (as right modules)}$$ implies that $m=n$.
		
		(2) (\cite[page 282]{hazrat:ggklpa} and \cite[Definition 4.3]{IBNgraded})
		A $\Gamma$-graded ring $R$ is said to have \emph{Graded Invariant Basis Number} (for short, gr-IBN) if,
		for any two systems of elements $\{\alpha_1, \dots, \alpha_m\}$ and $\{\beta_1, \dots, \beta_n\}$ of $\Gamma$, $$R(\alpha_1)\oplus  \cdots \oplus R(\alpha_m) \cong_{gr} R(\beta_1)\oplus  \cdots \oplus R(\beta_n) \text{ (as right modules)}$$ implies that $m=n$.
	\end{defn}
		Clearly, if $f$ is a graded module isomorphism then $f$ is also a module isomorphism (by forgetting the grading). It follows that:
	\begin{cor}\label{UGNIBNgrIBN} 
		Let $R$ be a $\Gamma$-graded ring. If $R$ has IBN, then $R$ has gr-IBN.
	\end{cor}
	
	In practice, it is not easy to determine the gr-IBN property of a given ring. However,the following lemma provides a criterion based on rings whose gr-IBN property is already known. This is the graded analogue of \cite[Remark 1.5]{l:lomar}.
	\begin{lem}\label{r--s}
			Suppose there is a graded ring homomorphism $R \to S$ and $S$ has gr-IBN. Then $R$ has gr-IBN. 
	\end{lem}
	\begin{proof}
		Assume that there are two systems of elements $\{\alpha_1, \dots, \alpha_m\}$ and $\{\beta_1, \dots, \beta_n\}$ of $\Gamma$ such that 
		$$
		R(\alpha_1)\oplus  \cdots \oplus R(\alpha_m) \cong_{gr} R(\beta_1)\oplus  \cdots \oplus R(\beta_n).
		$$
		Since $S$ can be regarded as a graded $R$-bimodule, we obtain
		$$
		\left(R(\alpha_1)\oplus  \cdots \oplus R(\alpha_m)\right) \otimes_R S \cong \left(R(\beta_1)\oplus  \cdots \oplus R(\beta_n)\right) \otimes_R S.
		$$
		Hence,
		$$
		S(\alpha_1)\oplus  \cdots \oplus S(\alpha_m) \cong_{gr} S(\beta_1)\oplus  \cdots \oplus S(\beta_n).
		$$
		Because $S$ has gr-IBN, it follows that $m = n$.
	\end{proof}

	We immediately identify some classes of graphs that the Leavitt path algebra associated have gr-IBN.
	Let $E$ be a graph. For a cycle $c$ and a sink $v$ in $E$, we write $c\Rightarrow v$ if there exists a path in $E$ which starts in $c$ and ends at $v$. A sink $v$ in $E$ is called {\it maximal} if there is no a cycle $c$ in $E$ such that $c\Rightarrow v$. A cycle $c$ in $E$ is called {\it maximal} if there is no a cycle $d$ in $E$ which is different from a cyclic permutation of $c$ such that $d\Rightarrow c$. The {\it predecessors} of a vertex $v$ in $E$ is the set $E_{\ge v} := \{w\in E^0\mid w\ge v\}$ and the {\it predecessors} of a cycle $c$ in $E$ is the set $E_{\ge v}$, where $v$ is an arbitrary vertex on $c$.
	\begin{cor}\label{maximal}
		Let $E$ be a row-finite graph. If $E$ has a maximal sink or cycle with finitely many predecessors then $L_K(E)$ has gr-IBN.
	\end{cor}
	\begin{proof}
		By \cite[Corollary 3.9]{np:tsolpaatibnp}, $L_K(E)$ has IBN. Then, Corollary \ref{UGNIBNgrIBN} leads to $L_K(E)$ having gr-IBN.
	\end{proof}
	Using Corollary \ref{maximal}, we imply the gr-IBN property of the Leavitt path algebra of a finite graph in which every vertex is in at most one cycle. In \cite{aajz:lpaofgkd} the authors showed that the Leavitt path
	algebra of such a graph has finite Gelfand-Kirillov dimension and vice versa. 
	
	\begin{cor}\label{atmost}
		Let $K$ be a field and $E$ a finite graph in which every vertex is in at most one cycle.	Then $L_K(E)$ has gr-IBN.
	\end{cor}
	
	\subsection{Graph monoids and IBN criterion}
	
	We recall the graph monoid according to \cite{amp:nktfga} and a criterion for $L_K(E)$ having IBN in \cite{np:tsolpaatibnp}.
	\begin{defn}[{\cite[page 163]{amp:nktfga}}]\label{graph}
	Let $E$ be a row-finite graph. The \emph{graph monoid} of $E$, denoted $M_{E}$, is the commutative monoid generated by $\left\{v \mid v \in E^{0}\right\}$, subject~to
 	\begin{equation*}
 	v=\sum_{e \in s^{-1}(v)} r(e)
	 \end{equation*}
	for every $v \in E^{0}$ that is not a sink.
	\end{defn}
	\begin{lem}[{\cite[Lemma 4.3]{amp:nktfga}, Confluence Lemma}] \label{Confluence}
		Let $E$ be a row-finite graph. Then, $a=b$ in $M_{E}$ if and only if there exists $c \in F_{E}$ such that $a \rightarrow c$ and $b \rightarrow c$. $($Note that, in this case, $a=b=c$ in $M_{E}$.$)$
	\end{lem}
	
	Ara et al \cite{amp:nktfga} showed that $\mathcal{V}(L_K(E))\cong M_E$. Based on this result, Nam and the author give a criterion for Leavitt path algebras of finite graphs to have IBN in~\cite{np:tsolpaatibnp}.
	
	\begin{thm}[{\cite[Theorem 3.1]{ko:clpaatibnp} and \cite[Theorem 3.5]{np:tsolpaatibnp}}]\label{IBN}
		Let $E$ be a finite graph having vertices $\{v_1, v_2, \cdots, v_h\}$ such that the regular
		vertices appear as $v_1,..., v_z$. Let 
		\[J_E = \left(
		\begin{array}{cc}
			I_z & 0 \\
			0 & 0 \\
		\end{array}
		\right)
		\in M_{h}(\mathbb{N})\ \text{ and }\ b= (1 \ ...\ 1)^t
		\in M_{h\times 1}(\mathbb{N}),\] and $[A^t_E - J_E \ \ b]$ the matrix obtained from
		the matrix $A^t_E - J_E$ by adding the column $b$. Then $L_K(E)$ has IBN if and only if \[\mathrm{rank}(A^t_E - J_E)<
		\mathrm{rank}([A^t_E - J_E \ \ b]).\]
	\end{thm}
%The relations defining $M_{E}$ can be described more concretely as follows: We denote by $F_{E}$ the free commutative monoid generated by $E^{0}$.
% First, define a relation $\rightarrow_{1}$ on $F_{E}$ by setting, for each $v \in E^{0}$,
%
%$$
%v \rightarrow_{1} \sum_{e \in s^{-1}(v)} r(e) .
%$$
%
%Then $M_{E}$ is the quotient of $F_{E}$ by the congruence generated by $\rightarrow_{1}$.
%
%Let $\rightarrow$ be the smallest reflexive, transitive and additive relation on $F_{E}$ which contains (is coarser than) $\rightarrow_{1}$. Note that $\rightarrow$ is not symmetric, so it is not a congruence.

\subsection{ The talented monoid and covering graph}

We recall the talented monoid $T_{E}$ of $E$, which encodes the graded structure of a Leavitt path algebra $L_K(E)$ as~well.

\begin{defn}[{\cite[page 436]{hl:tmLpa}}]\label{talent}
Let $E$ be a row-finite directed graph. The talented monoid of $E$, denoted $T_{E}$, is the commutative monoid generated by $\left\{v(\alpha) \mid v \in E^{0}, \alpha \in \mathbb{Z}\right\}$, subject to

\begin{equation}\label{TE}
v(\alpha)=\sum_{e \in s^{-1}(v)} r(e)(\alpha+1)	
\end{equation}
for every $\alpha \in \mathbb{Z}$ and every regular vertex $v\in E^0$.
\end{defn}
The additive group $\mathbb{Z}$ of integers acts on $T_{E}$ via monoid automorphisms by shifting indices: For each $n, \alpha \in \mathbb{Z}$ and $v \in E^{0}$, define ${ }^{n} v(\alpha)=v(\alpha+n)$, which extends to an action of $\mathbb{Z}$ on $T_{E}$.

The talented monoid of a graph can also be seen as a special case of a graph monoid, which we now describe. The {\it covering graph} of $E$ is the graph $\overline E$ with vertex set $\overline E^{0}=E^{0} \times \mathbb{Z}$, and edge set $\overline E^{1}=E^{1} \times \mathbb{Z}$. The range and source maps are given as

$$
s(e, \alpha)=(s(e), \alpha), \quad r(e, \alpha)=(r(e), \alpha+1) .
$$

Note that the graph monoid $M_{\overline E}$ has a natural $\mathbb{Z}$-action by ${ }^{n}(v, \alpha)=(v, \alpha+n)$. 

The following proposition justifies our use of the talented monoid as the main tool for studying gr-IBN: it shows that graded module isomorphisms in $\mathcal{V}^{\text{gr}}(L_K(E))$ correspond precisely to equalities in $T_E$.
\begin{prop}[{\cite[Proposition 5.7]{Ara-Hazrat-Li-Sims} and \cite[Lemma 3.2]{hl:tmLpa}}]\label{TEME}
	Let $E$ be an arbitrary graph. Then there are $\mathbb{Z}$-monoid isomorphisms:
	$$
	\begin{array}{ccccc}
		T_{E} & \cong & M_{\overline E} & \cong & \mathcal{V}^{\text{gr}}(L_K(E))  \\
		v(\alpha) & \longmapsto &(v, \alpha)&\longmapsto&L_K(E)v(\alpha).\\
	\end{array}
	$$
\end{prop}

By Proposition \ref{TEME}, throughout the paper we freely identify elements of $T_E$ with their images in $M_{\overline E}$. A monoid is called {\it conical} if $a + b = 0$ implies $a = b = 0$ and it is called {\it cancellative} if $a + b = a + c$
implies $b = c$. By \cite[Corollary 5.8]{h:tgslpa}, $T_{E}$ is also conical and cancellative. This fact may also be proved directly using the Confluence Lemma \ref{Confluence}.
%Consequently, results on $M_{\overline E}$, such as Lemma \ref{Confluence}, also hold for $T_E$.

\begin{exas}\label{exTE}
Consider the following graphs 
\medskip $$\xymatrix{E:& \bullet_u \ar@(dl,ul)^e \ar@/^.5pc/[r]^f& \bullet_v \ar@/^.5pc/[l]^g}$$ 
\medskip
Then, the covering graph is
$$
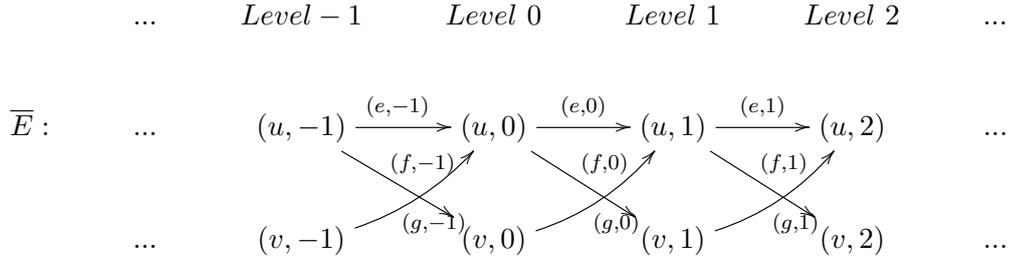
\xymatrix{& ... & Level -1 & Level\ 0 & Level\ 1 & Level\ 2 & ...\\
	\overline E:& ... & (u,-1) \ar[r]^{(e,-1)} \ar[dr]^{(f,-1)}   & (u,0) \ar[r]^{(e,0)} \ar[dr]^{(f,0)}     &  (u,1) \ar[r]^{(e,1)} \ar[dr]^{(f,1)}    & (u,2)    &  ...\\
	& ... & (v,-1) \ar@/_.7pc/[ur]_{(g,-1)}   & (v,0) \ar@/_.7pc/[ur]_{(g,0)}      &  (v,1) \ar@/_.7pc/[ur]_{(g,1)}    & (v,2)   &  ...}$$

The relation (\ref{TE}) gives that $u(1)=u(2)+v(2)$, $v(1)=u(2)$. Hence, 
$$
u(1)+v(1)=2u(2)+v(2)\text{ in } T_E.
$$
In other words, 
\begin{equation*}
	u(1)+v(1)=(1, 1)_{1\times 2}\left(\begin{tabular}{cc}
		1 & 1\\
		1 & 0\\
	\end{tabular}\right)\left(\begin{tabular}{c}
		$u(2)$\\
		$v(2)$\\
	\end{tabular}\right)=(1, 1)_{1\times 2}A_E\left(\begin{tabular}{c}
		$u(2)$\\
		$v(2)$\\
	\end{tabular}\right),
\end{equation*} where $(1, 1)_{1\times 2}$ denotes the all-ones row vector of size $2$ and $A_E$ is the adjacency matrix of $E$. \hfill $\Box$ \end{exas}
%	By iterating Lemma \ref{write}, in the absence of sinks, elements of the talented monoid can be shifted to arbitrarily higher levels. This observation will be used extensively in proving the main results of this paper in the next sections.

%Note that $M_{E}$ is the quotient of $T_{E}$ obtained by identifying elements of $T_{E}$ which belong to the same $\mathbb{Z}$-orbit. The respective quotient map,
%
%$$
%\begin{aligned}
%	T_{E} & \longrightarrow M_{E} \\
%	v(i) & \longmapsto v,
%\end{aligned}
%$$
%is also called the forgetful homomorphism. It follows that any $\mathbb{Z}$-monoid homomorphism between $T_{E}$ and $T_{F}$, for row-finite graphs $E$ and $F$, induces a monoid homomorphism between $M_{E}$ and $M_{F}$.

%We note that if $\phi: E \rightarrow F$ is a complete graph homomorphism, then $\phi$ extends to a natural $\mathbb{Z}$-monoid homomorphism $\bar{\phi}: T_{E} \rightarrow T_{F}$. In the case of Leavitt path algebras, the map $\phi$ induces an injective ring homomorphism $\bar{\phi}: L_{\mathrm{k}}(E) \rightarrow L_{\mathrm{k}}(F)$. However injectivity does not follow in the setting of talented monoids, as the following example shows. For the graphs $E$ and $F$,
%the $\mathbb{Z}$-monoid homomorphism $\bar{\phi}: T_{E} \rightarrow T_{F}$ is not injective, as in $T_{E}$ we have $u \neq u(2)+u(2)$, whereas their images under $\bar{\phi}$ coincide.

%\section{Leavitt path algebras having IBN}
%BO MUC NAY, NOI NGAN GON TRONG MUC 5 HOAC MUC GIOI THIEU

\section{Leavitt path algebras having Graded Invariant Basis Number}
	In this section, we study Graded Invariant Basis Number for Leavitt path
	algebras via the talented monoid. The main goal is to establish a matrix-theoretic characterization of when this property fails, which is achieved in
	Theorem \ref{thm-nogribn}.

%
%\begin{defn}[{\cite[Definition 3.1]{IBNgraded}}]
%	A unital ring $R$ is said to have \emph{graded Invariant Basis Number} (for short, gr-IBN) if,
%	for any pair of positive integers $m \text{ and } n$, $R^m \cong_{gr} R^n$ (as right modules) implies that $m=n$.
%\end{defn}

As mentioned in Section 2, we have a natural $\mathbb{Z}$-grading on $L_K(E)$. Then, we have the graded version of \cite[Lemma 3.4]{np:tsolpaatibnp}, which reduces the problem of gr-IBN for $L_K(E)$ to the study of equalities in the talented monoid $T_E$.
%Here, $\left[\sum\limits_{i=1}^hv_i(\alpha_j)\right]$ corresponds to the class of the graded free module $\left[L_K(E)(\alpha_j)\right]$ in $T_E$.

\begin{lem}\label{gr-language}
	Let $E$ be a finite graph with vertex set $\{v_1, v_2, \cdots, v_h\}$. Then, 
	$L_K(E)$ has gr-IBN if and only if the following holds: for any pair of positive integers $m$, $n$, and any $\alpha_1,\cdots, \alpha_m, \beta_1, \cdots, \beta_n\in \mathbb{Z}$
		the equality $$\sum\limits_{j=1}^m\left(\sum\limits_{i=1}^hv_i(\alpha_j)\right)=\sum\limits_{l=1}^n\left(\sum\limits_{k=1}^hv_k(\beta_l)\right)
		\ \text{in } T_{E}
		$$
		implies $m = n$.
\end{lem}
\begin{proof}
Since $E$ is finite graph, $1_{L_K(E)}=\sum\limits_{i=1}^hv_i$. So, $\sum\limits_{i=1}^hv_i$ represents the identity element $\left[L_K(E)\right]$ in $\mathcal{V}(L_K(E))$, 
and $\sum\limits_{i=1}^hv_i(\alpha_j)$ represents the class $\left[L_K(E)(\alpha_j)\right]$ in $\mathcal{V}^{\text{gr}}(L_K(E))$.
	
	Assume that 
	\begin{equation*}
		L_K(E)(\alpha_1)\oplus \cdots \oplus L_K(E)(\alpha_m)\cong_{gr}L_K(E)(\beta_1)\oplus \cdots \oplus L_K(E)(\beta_n)
	\end{equation*} for some $m, n\in \mathbb{N}$ and $\alpha_j, \beta_l\in \mathbb{Z}, 1\leq j \leq m, 1\leq l \leq n.$ Equivalently, we have $$[L_K(E)(\alpha_1)]+\cdots + [L_K(E)(\alpha_m)]=[L_K(E)(\beta_1)]+\cdots + [L_K(E)(\beta_n)] \text{ in } \mathcal{V}^{\text{gr}}(L_K(E)).$$ Equivalently, we have
	$$\sum\limits_{j=1}^m\left(\sum\limits_{i=1}^hv_i(\alpha_j)\right)=\sum\limits_{l=1}^n\left(\sum\limits_{k=1}^hv_k(\beta_l)\right) \text { in } T_{E}.$$ Conversely, by Proposition \ref{TEME}, any equality of the above form in $T_E$  corresponds to a graded module isomorphism	
	$$
	L_K(E)(\alpha_1)\oplus \cdots \oplus L_K(E)(\alpha_m)\cong_{gr}L_K(E)(\beta_1)\oplus \cdots \oplus L_K(E)(\beta_n).
	$$	
	Thus finishing the proof.
\end{proof}
The following result indicates that the Leavitt path algebra of graphs with sinks always has gr-IBN. And so, in investigating the gr-IBN property of Leavitt path algebras, we need only concentrate on the graphs without sinks.
\begin{prop}\label{hassink}
	Let $E$ be a finite graph. If $E$ has a sink, then $L_K(E)$ has gr-IBN.
\end{prop}
\begin{proof}
	Assume that $\{v_1, v_2, \cdots, v_h\}$ is the set of vertices of $E$, where $v_1$ is a sink, and 
	\begin{equation}\label{sink}
		\sum\limits_{j=1}^m\left(\sum\limits_{i=1}^hv_i(\alpha_j)\right)=\sum\limits_{l=1}^n\left(\sum\limits_{k=1}^hv_k(\beta_l)\right) \text{ in } T_{E} \text{ for some } \alpha_i, \beta_l\in \mathbb{Z}.
	\end{equation} 
	Since $T_E$ is a commutative monoid, we may rearrange the sums and assume without loss of generality that
	$\alpha_1\leq \cdots \leq \alpha_m$, $\beta_1\leq \cdots \leq \beta_n$.% and $\alpha_m \leq \beta_n$.
	
	Claim: $\alpha_1=\beta_1$. 
	Since $v_1$ is a sink, the relation (\ref{TE}) does not apply to $v_1$, that is, $v_1$ cannot be written as a sum of other generators at higher levels. More precisely, for any $\alpha \in \mathbb{Z}$, the element $v_1(\alpha)$ cannot be transformed in $T_E$ in the sense that it does not appear on the left-hand side of any relation (\ref{TE}).
	
	Now consider the element $v_1(\alpha_1)$ in the left-hand side of equation (\ref{sink}). Since $\alpha_1$ is the smallest level appearing on the left-hand side, the element $v_1(\alpha_1)$ cannot be obtained by applying relations (\ref{TE}) to any other generator. So, $v_1(\alpha_1)$ in the left-hand side cannot be changed.
	
	Similarly, in the right-hand side $v_1(\beta_1)$ appears at level $\beta_1$  (the smallest level on the right-hand side), and it also cannot be change. So, we must have $\alpha_1=\beta_1$  and the coefficient of $v_1(\alpha_1)$ on both sides must be equal (which is 1 in this case).
	
%	Since $v_1$ is a sink, the relation in Definition 3.3 is not applied for $v_1(\alpha_1)$ and $v_1(\beta_1)$. In the left side, at level $\alpha_1$, there is only an element $v_1(\alpha_1)$. And the coefficient of $v_1(\alpha_1)$ in the left side can not increase. This is similar for $v_1(\beta_1)$ in the right side. So, (\ref{sink}) implies that $\alpha_1=\beta_1$.
	
	%Since $T_E$ is cancellative, we obtain that 
	
Since $T_E$ is cancellative (by \cite[Corollary 5.8]{h:tgslpa}), we can cancel $\sum_{i=1}^hv_i(\alpha_1)$ from both sides of equation (\ref{sink}) to obtain:
	\begin{equation*}
		\sum\limits_{j=2}^m\left(\sum\limits_{i=1}^hv_i(\alpha_j)\right)=\sum\limits_{l=2}^n\left(\sum\limits_{k=1}^hv_k(\beta_l)\right) \text{ in } T_{E}.
	\end{equation*}
Repeating this process inductively, we obtain 
$
\alpha_2 = \beta_2,\ \alpha_3 = \beta_3,\ \text{and so on}.
$
Eventually, we must have $m = n$ (otherwise one side would be exhausted before the other, yielding a non-trivial element equal to zero in $T_E$, which is impossible since $T_E$ is conical).
Therefore, $L_K(E)$ has gr-IBN by Lemma \ref{gr-language}. Thus  finishing the proof.
\end{proof}

We illustrate the construction in the proof of Proposition \ref{hassink} with a specific example.
\begin{exas}\label{exgribn}
	Consider the following graph:
	
	\medskip
	$$
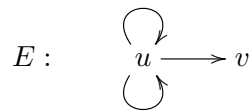
\xymatrix{E: & u \ar@(ul,ur) \ar@(dl,dr) \ar[r]& v}$$
	\medskip
 
	We show that $L_K(E)$ has gr-IBN.
	Assume that
	
	\begin{equation}\label{exam5.7}
	\sum_{j=1}^{m}
	\left(
	u(\alpha_j)+v(\alpha_j)
	\right)
	=
	\sum_{l=1}^{n}
	\left(
	u(\beta_l)+v(\beta_l)
	\right)
	\end{equation}
	in $T_E$, where $\alpha_1,\ldots,\alpha_m,\beta_1,\ldots,\beta_n \in \mathbb{Z}$ are arbitrary.
	Without loss of generality, we may assume that
	$\alpha_1 \le \cdots \le \alpha_m$ and
	$\beta_1 \le \cdots \le \beta_n$.
	
	Then, the covering graph is
	$$
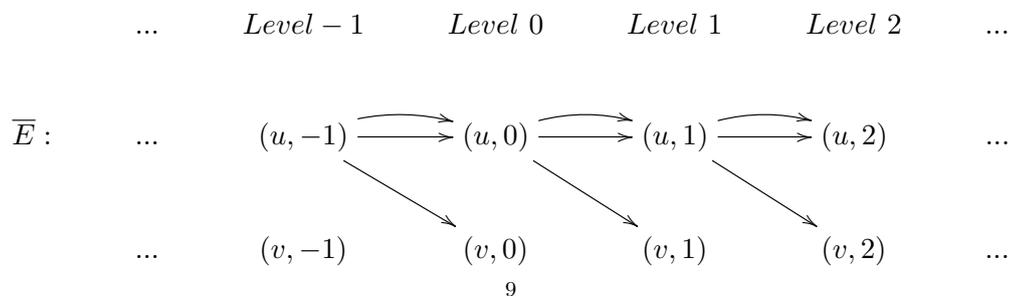
\xymatrix{& ... & Level -1 & Level\ 0 & Level\ 1 & Level\ 2 & ...\\
		\overline E:& ... & (u,-1) \ar@/^.7pc/[r] \ar[r] \ar[dr]   & (u,0) \ar@/^.7pc/[r] \ar[r] \ar[dr]     &  (u,1) \ar@/^.7pc/[r] \ar[r] \ar[dr]    & (u,2)    &  ...\\
		& ... & (v,-1)    & (v,0)       &  (v,1)     & (v,2)   &  ...}$$
	
	The relation (\ref{TE}) gives that $u(\alpha)=2u(\alpha+1)+v(\alpha+1)$  in $T_E$ for all $\alpha \in \mathbb{Z}$.
	
	Since $v$ is a sink, the element $v(\alpha_1)$ appears at the smallest
	level on the left-hand side of (\ref{exam5.7}) and cannot be obtained by applying any relation (\ref{TE}) in the definition of $T_E$.
	Similarly, $v(\beta_1)$ on the right-hand side of (\ref{exam5.7}) cannot be transformed.
	Therefore, we must have $\alpha_1 = \beta_1$.
	
	By the cancellativity of $T_E$,
	we can cancel $u(\alpha_1)+v(\alpha_1)$ from both sides of (\ref{exam5.7}).
	Continuing inductively, we obtain
	$\alpha_2 = \beta_2$, $\alpha_3 = \beta_3$, and so on.
	If $m \ne n$, say $m < n$, then after canceling all $m$ terms on the left-hand side, at least one nonzero term remains on the right-hand side, contradicting the equality. Hence, $m = n$. Therefore, by Lemma \ref{gr-language}, $L_K(E)$ has gr-IBN. \hfill $\Box$
\end{exas}

	The following example shows that, within the class of Leavitt path algebras, the gr-IBN property is strictly weaker than the IBN property.
\begin{exas}\label{ex3.7}
	Consider the graph:
	\medskip
	$$
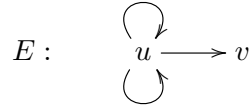
\xymatrix{E: & u \ar@(ul,ur) \ar@(dl,dr) \ar[r]& v}$$
	\medskip
	
	Then, it is easy to see that
	$u+v=2(u+v)$
	in $M_E$. Hence, 
	$L_K(E)\cong L_K(E)^2$, and so, $L_K(E)$ has no IBN. By Theorem \ref{IBN}, we can also verify that $L_K(E)$ has no IBN by checking the following equality. 
	$$
	\mathrm{rank}(A^t_E - J_E)=
	\mathrm{rank}([A^t_E - J_E \ \ b])=1.
	$$ 
	But from Example \ref{exgribn}, $L_K(E)$ has gr-IBN. \hfill $\Box$
\end{exas}

Generalizing Example \ref{exTE}, we arrive at the following useful assertion.

\begin{lem} \label{write}
	Let $E$ be a finite graph without sinks, with vertex set
	$$
	E^0 = \{v_1,\dots,v_h\},
	$$
	and adjacency matrix $A_E$. Then, in $T_E$ we have
	\begin{equation*}
		v_1(\alpha)+v_2(\alpha)+\cdots +v_h(\alpha)=(1, 1, \cdots, 1)_{1\times h}A_E\left(\begin{tabular}{c}
			$v_1(\alpha +1)$\\
			$v_2(\alpha +1)$\\
			$\cdots$\\
			$v_h(\alpha +1)$\\
		\end{tabular}\right)
	\end{equation*} for all $\alpha \in \mathbb{Z}$,  where $(1, 1, \cdots, 1)_{1\times h}$ denotes the all-ones row vector of size $h$.
\end{lem} 
\begin{proof}
	Let $A_E=\left(a_{ij}\right)_h$ be the adjacency matrix of $E$. For all $\alpha \in \mathbb{Z}$, the relation (\ref{TE}) gives that 
	$$
	v_i(\alpha)=a_{i1}v_1(\alpha+1)+a_{i2}v_2(\alpha+1)\cdots + a_{ih}v_h(\alpha+1),
	$$
	for each $1\leq i \leq h$. Summing over all $i$, the expression $v_1(\alpha)+v_2(\alpha)+\cdots +v_h(\alpha)$ 
	is equal to
	\begin{equation*}
		\begin{array}{c}
			\begin{array}{c}
				a_{11}v_1(\alpha+1)+a_{12}v_2(\alpha+1)\cdots + a_{1h}v_h(\alpha+1)+\\
				a_{21}v_1(\alpha+1)+a_{22}v_2(\alpha+1)\cdots + a_{2h}v_h(\alpha+1)+\\
				\cdots\\
				a_{h1}v_1(\alpha+1)+a_{h2}v_2(\alpha+1)\cdots + a_{hh}v_h(\alpha+1).\\
			\end{array}\\
		\end{array}
	\end{equation*}
	In other words,
	\begin{equation*}
		\begin{array}{rcl}
			v_1(\alpha)+v_2(\alpha)+\cdots +v_h(\alpha)&=&(1, 1, \cdots, 1)_{1\times h}A_E\left(\begin{tabular}{c}
				$v_1(\alpha +1)$\\
				$v_2(\alpha +1)$\\
				$\cdots$\\
				$v_h(\alpha +1)$\\
			\end{tabular}\right),\\
		\end{array}
	\end{equation*}
	where $(1, 1, \cdots, 1)_{1\times h}$ denotes the all-ones row vector of size $h$.
\end{proof}

We can now state and prove the main result of this paper, providing a characterization of the graphs whose associated Leavitt path algebras fail to have gr-IBN.
\begin{thm}\label{thm-nogribn}
	Let $E$ be a finite graph without sinks with adjacency matrix $A_E$. Then $L_K(E)$ has no gr-IBN if and only if there exist positive integers $m\neq n$ and   nonnegative integers $p_1, \dots, p_m, q_1, \dots, q_n$ such that
	\begin{equation*}
		(1, 1, \cdots, 1)(A_E^{p_1}+\cdots +A_E^{p_m})=(1, 1, \cdots, 1)(A_E^{q_1}+\cdots +A_E^{q_n})
	\end{equation*} where $(1, 1, \cdots, 1)$ denotes the all-ones row vector.
\end{thm}
\begin{proof}
	Let $E$ be a finite graph without sinks, with vertex set
	$$
	E^0 = \{v_1,\dots,v_h\},
	$$
	and adjacency matrix $A_E$. 
		
	($\Longrightarrow$)
	Assume that $L_K(E)$ does not have gr-IBN.
	By Lemma \ref{gr-language}, there exist positive integers $m < n$ and integers
	$\alpha_1,\dots,\alpha_m,\beta_1,\dots,\beta_n \in \mathbb{Z}$ such that
	
	\begin{equation*}
		\sum_{j=1}^{m}\left(\sum_{i=1}^{h} v_i(\alpha_j)\right) = \sum_{l=1}^{n}\left(\sum_{k=1}^{h} v_k(\beta_l)\right) \text{ in } T_E.
	\end{equation*}	
%	By Lemma \ref{TEME}, we imply that 
%		\begin{equation*}
%		\left[\sum_{j=1}^{m}\left(\sum_{i=1}^{h} (v_i,\alpha_j)\right)\right] = \left[\sum_{l=1}^{n}\left(\sum_{k=1}^{h} (v_k,\beta_l)\right)\right]
%	\end{equation*}
%	in $M_{\overline{E}}$. 
	By Lemma \ref{Confluence}, there exists $\gamma$ in $T_E$ such that 
	\begin{equation*}
		\sum_{j=1}^{m}\left(\sum_{i=1}^{h} v_i(\alpha_j)\right)\rightarrow \gamma \text{ and }
\sum_{l=1}^{n}\left(\sum_{k=1}^{h} v_k(\beta_l)\right)\rightarrow \gamma
	\end{equation*}
	The defining
	relations of $T_E$ allow us to write $\gamma \in T_E$ as $$\gamma=\lambda_1v_1(\delta_1)+\cdots+\lambda_hv_h(\delta_h)$$
	where $\lambda_1, \dots, \lambda_h$ are nonnegative integers and $\delta_j\geq \max\{\alpha_1, \cdots , \alpha_m, \beta_1, \cdots , \beta_n\}$. Let $\delta=\max\{\delta_i, 1\leq i \leq h\}$. Since $E$ has no sinks, for each $v_i(\delta_i)$ with $\delta_i < \delta$, we may repeatedly apply relation~(\ref{TE}) to expand $v_i(\delta_i)$ into a linear combination of generators at level $\delta_i + 1$, then $\delta_i + 2$, and so on, until all generators are at level $\delta$. 
	Then, we obtain
	$$\gamma\rightarrow \gamma'=\lambda'_1v_1(\delta)+\cdots+\lambda'_hv_h(\delta)$$ where $\lambda_1', \dots, \lambda_h'$ are nonnegative integers. Equivalently, in $T_E$ we have
	\begin{equation}\label{ME}
		\sum_{j=1}^{m}\left(\sum_{i=1}^{h} v_i(\alpha_j)\right)=\lambda'_1v_1(\delta)+\cdots+\lambda'_hv_h(\delta)= \sum_{l=1}^{n}\left(\sum_{k=1}^{h} v_k(\beta_l)\right)
	\end{equation}

	Using the Lemma \ref{write}, we have that
	\begin{equation*}
	v_1(\alpha)+v_2(\alpha)+\cdots +v_h(\alpha)=(1, 1, \cdots, 1)_{1\times h}A_E\left(\begin{tabular}{c}
		$v_1(\alpha +1)$\\
		$v_2(\alpha +1)$\\
		$\cdots$\\
		$v_h(\alpha +1)$\\
	\end{tabular}\right)
	\end{equation*} for all $\alpha \in \mathbb{Z}$. Iterating this relation $k$ times, we obtain
		\begin{equation*}
		v_1(\alpha)+v_2(\alpha)+\cdots +v_h(\alpha)=(1, 1, \cdots, 1)_{1\times h}A^k_E\left(\begin{tabular}{c}
			$v_1(\alpha +k)$\\
			$v_2(\alpha +k)$\\
			$\cdots$\\
			$v_h(\alpha +k)$\\
		\end{tabular}\right).
	\end{equation*}

		Let $p_j=\delta-\alpha_j$, $q_l=\delta-\beta_l$ for all $1\leq j \leq m, 1\leq l \leq n$, from (\ref{ME}) we get that 
			$$\sum_{j=1}^{m}\left(\sum_{i=1}^{h} v_i(\alpha_j)\right)= (1, 1, \cdots, 1)_{1\times h}\left(\sum_{j=1}^{m} A_E^{p_j}\right)\left(\begin{tabular}{c}
				$v_1(\delta)$\\
				$v_2(\delta)$\\
				$\cdots$\\
				$v_h(\delta)$\\
			\end{tabular}\right) =\lambda'_1v_1(\delta)+\cdots+\lambda'_hv_h(\delta)$$
		
		and similarly:
				$$\sum_{l=1}^{n}\left(\sum_{k=1}^{h} v_k(\beta_l)\right)= (1, 1, \cdots, 1)_{1\times h}\left(\sum_{l=1}^{n} A_E^{q_l}\right)\left(\begin{tabular}{c}
			$v_1(\delta)$\\
			$v_2(\delta)$\\
			$\cdots$\\
			$v_h(\delta)$\\
		\end{tabular}\right) = \lambda'_1v_1(\delta)+\cdots+\lambda'_hv_h(\delta)$$
		
		This implies:
	\begin{equation*}
	(1, 1, \cdots, 1)_{1\times h}(A_E^{p_1}+\cdots +A_E^{p_m})=(1, 1, \cdots, 1)_{1\times h}(A_E^{q_1}+\cdots +A_E^{q_n}),
 	\end{equation*}  as desired.
 	
 		($\Longleftarrow$) Assume that there exist $p_1, \ldots, p_m, q_1, \ldots, q_n \in \mathbb{N}$ with $m \neq n$ such that
 	$$(1, 1, \cdots, 1)(A_E^{p_1} + \cdots + A_E^{p_m}) = (1, 1, \cdots, 1)(A_E^{q_1} + \cdots + A_E^{q_n}).$$
 	
 	Let $\delta = \max\{p_1, \ldots, p_m, q_1, \ldots, q_n\}$. Set $\alpha_j = \delta - p_j$ for $j = 1, \ldots, m$ and $\beta_l = \delta - q_l$ for $l = 1, \ldots, n$. 
 	We show that 
 	$$\sum_{j=1}^{m}\left(\sum_{i=1}^{h} v_i(\alpha_j)\right) = \sum_{l=1}^{n}\left(\sum_{k=1}^{h} v_k(\beta_l)\right) \text{ in } T_E.$$

 	Applying the Lemma \ref{write} repeatedly, for any $k \geq 0$:
	\begin{equation*}
	v_1(\alpha)+v_2(\alpha)+\cdots +v_h(\alpha)=(1, 1, \cdots, 1)_{1\times h}A^k_E\left(\begin{tabular}{c}
		$v_1(\alpha +k)$\\
		$v_2(\alpha +k)$\\
		$\cdots$\\
		$v_h(\alpha +k)$\\
	\end{tabular}\right)
\end{equation*} for all $\alpha \in \mathbb{Z}$.

 	Therefore, in $T_E$, for each $j = 1, \ldots, m$:
 	$$\sum_{i=1}^{h} v_i(\alpha_j) = \sum_{i=1}^{h} v_i(\delta - p_j) = (1, 1, \cdots, 1)_{1\times h} A_E^{p_j} \begin{pmatrix} v_1(\delta) \\
 		v_2(\delta)\\ \cdots \\ v_h(\delta) \end{pmatrix}.$$
 	
 	Summing over all $j$:
 	$$\sum_{j=1}^{m}\left(\sum_{i=1}^{h} v_i(\alpha_j)\right) = (1, 1, \cdots, 1)_{1\times h} \left(\sum_{j=1}^{m} A_E^{p_j}\right) \begin{pmatrix} v_1(\delta)\\v_2(\delta) \\ \cdots \\ v_h(\delta) \end{pmatrix}.$$
 	
 	Similarly, for the right-hand side:
 	$$\sum_{l=1}^{n}\left(\sum_{k=1}^{h} v_k(\beta_l)\right) = (1, 1, \cdots, 1)_{1\times h} \left(\sum_{l=1}^{n} A_E^{q_l}\right) \begin{pmatrix} v_1(\delta)\\v_2(\delta) \\ \cdots \\ v_h(\delta) \end{pmatrix}.$$
 	
 	By hypothesis, 
 	$$
 	(1, 1, \cdots, 1)_{1\times h}(A_E^{p_1} + \cdots + A_E^{p_m}) = (1, 1, \cdots, 1)_{1\times h}(A_E^{q_1} + \cdots + A_E^{q_n})
 	$$
 	thus
 	$$\sum_{j=1}^{m}\left(\sum_{i=1}^{h} v_i(\alpha_j)\right) = \sum_{l=1}^{n}\left(\sum_{k=1}^{h} v_k(\beta_l)\right) \text{ in } T_E.$$
 	Since $m \neq n$, by Lemma \ref{gr-language}, $L_K(E)$ has no gr-IBN. 	This completes the~proof. 	
\end{proof}

The following corollary characterizes a class of graphs satisfying the hypotheses of Theorem \ref{thm-nogribn}.
\begin{cor}\label{collum}
	Let $E$ be a finite graph without sinks having adjacency matrix $A_E=(a_{ij})_h$. If $$\sum\limits_{i=1}^ha_{i1}=\sum\limits_{i=1}^ha_{i2}=\cdots =\sum\limits_{i=1}^ha_{ih}\geq 2,$$ then $L_K(E)$ has no $($gr-$)$IBN.
\end{cor}
\begin{proof}
	Let $c=\sum\limits_{i=1}^ha_{i1}=\sum\limits_{i=1}^ha_{i2}=\cdots =\sum\limits_{i=1}^ha_{ih}$.
%	 We first claim that $c\geq 2$. Indeed, since $E$ has no sinks, if $c=1$ then $A_E$ is $$A_E=I_h=\left(\begin{tabular}{cccc}
%		1&0&...&0\\
%		0&1&...&0\\
%		...&...&...&...\\
%		0&0&...&1\\
%	\end{tabular}\right).$$
%	Then, $E$ is disjoint union of loops, a contradiction. 
	A direct computation shows that for any $k\geq 1$ : $(1, 1, \cdots, 1)_{1\times h}A^k=(c^k, c^k,  \cdots, c^k)_{1\times h}$, where $(x, x, \cdots, x)_{1\times h}$ denotes the row matrix with $h$ entries of $x$.
	Since $c \ge 2$, we have
	$$
	c^k = \underbrace{c^{k-1} + \cdots + c^{k-1}}_{c \text{ times}}.
	$$
	Therefore, let $p_1=k\geq 1, q_1, \dots q_r=k-1\in \mathbb{N}$, $r=c\geq 2$, we get that
	\begin{equation*}
		(1, 1, \cdots, 1)_{1\times h}A^{p_1}=(1, 1, \cdots, 1)_{1\times h}(A^{q_1}+\cdots +A^{q_r})
	\end{equation*}
	Hence, the condition of Theorem~5.9 is satisfied with $m \neq n$, and therefore
	$L_K(E)$ has no gr-IBN.
	It immediately follows from Corollary \ref{UGNIBNgrIBN} that  $L_K(E)$ has no IBN, finishing the proof.
%	On the other hand, we can also apply Theorem \ref{IBN} to show that $L_K(E)$ has no IBN since $\mathrm{rank}(A^t_E - J_E)=
%	\mathrm{rank}([A^t_E - J_E \ \ b])$. Indeed, since $$\sum\limits_{i=1}^ha_{i1}=\sum\limits_{i=1}^ha_{i2}=\cdots =\sum\limits_{i=1}^ha_{ih}\geq 2,$$ the column $b$ can be reduced to the sum of the columns of $(A^t_E - J_E)$.
\end{proof}

\begin{exas}
	\label{Ex:5.11}
	Consider the following graph
	\medskip
	$$
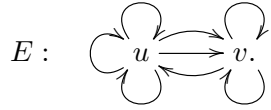
\xymatrix{E: & u \ar@(ld,lu) \ar@(ul,ur) \ar@(dr,dl) \ar[r] \ar@/^.7pc/[r]& v. \ar@(ul,ur) \ar@(dr,dl) \ar@/^.7pc/[l]}$$
	\medskip
	
	Then, the adjacency matrix of $E$ is given by
	$$
	A_E=\begin{pmatrix}
		3 & 2\\
		1 & 2
	\end{pmatrix}.
	$$
	
	A direct computation shows that $A_E$ satisfies the hypothesis of Corollary \ref{collum}, therefore  $L_K(E)$ has no gr-IBN. \hfill $\Box$ 
%	In addition, (\ref{ex12}) leads to 
%	\begin{equation*}
%		\left((A_E^t)^2-4A_E^t\right)\begin{pmatrix}
%			1\\
%			1
%		\end{pmatrix}=0,
%	\end{equation*}
%	equivalently, 
%	\begin{equation}\label{ex12c}
%		\left((A_E^t)^2-4A_E^t+I_2\right)b=b,
%	\end{equation}
%	where $b= (1 \ 1)^t$. Performing polynomial division gives us $$x^2-4x+1=(x-1)(x-3)-2.$$ Hence, $(A_E^t)^2-4A_E^t+I_2=(A_E^t-I_2)(A_E^t-3I_2)-2I_2$. Substitute into equation (\ref{ex12c}), we get that
%	\begin{equation*}
%		\left((A_E^t-I_2)(A_E^t-3I_2)\right)b=3b.
%	\end{equation*}
%	This show that $$x_0=\frac{1}{3}(A_E^t-3I_2)b=\begin{pmatrix}
%	1/3\\
%	1/3
%	\end{pmatrix}$$ is a solution of  the linear system
%	$$
%	(A_E^{t} - I_2)x = b.
%	$$
%	By the Kronecker–Capelli Theorem, it follows that
%	$$
%	\operatorname{rank}(A_E^{t} - I_2)
%	=
%	\operatorname{rank}\big([A_E^{t} - I_2\ \ b]\big).
%	$$
%	By Theorem \ref{IBN}, it shows that $L_K(E)$ has no IBN.
\end{exas}

\section{Applications}
In Section 3, we showed that the gr-IBN property is strictly weaker than the IBN property.  However, in this section, we show that these properties are equivalent on the Leavitt path algebras of Cayley graphs and Hopf graphs. We further examine the behavior of gr-IBN under quotients by hereditary saturated subsets and Cartesian products.

\subsection{Cayley graphs}

For any finite group $G$, and any set of generators $S$ of $G$, the \emph{Cayley graph of $G$ with respect to $S$} is the directed graph $E_{G, S}$ having the vertex set $\{v_g\mid g\in G\}$ and in which there is an edge from $v_g$ to $v_h$ in case there exists $s\in S$ such that $h=gs$ in $G$. In the particular case 
the Cayley graph $E_{G, S}$ of the cyclic group $G = \mathbb{Z}_n\ (n\geq 3)$ with respect to the generating set $S =\{1, j\}\ (0\leq j\leq n-1)$, denoted by $C^j_n$, is the directed graph consisting of $n$ vertices $\{v_1, v_2, ..., v_n\}$ and $2n$ edges $\{e_1, e_2, ..., e_n, f_1, f_2, ..., f_n\}$ for which $s(e_i) = v_i$, $r(e_i) = v_{i+1}$,
$s(f_i) = v_i$, and $r(f_i) = v_{i+j}$. For clarity, we picture the two graphs $C^0_4$ and $C^2_4$ here.
\bigskip
\medskip

\begin{tabular}{cccccccc}
	$C_4^0=$& &$\xymatrix{\bullet^{v_1} \ar[r] \ar@(lu,ld)&\bullet^{v_2} \ar[d] \ar@(ru,rd)\\
		\bullet^{v_4} \ar[u] \ar@(lu,ld)& \bullet^{v_3} \ar[l] \ar@(ru,rd)}$&\ &$C_4^1=$&$\xymatrix{\bullet^{v_1} \ar@/^.5pc/[r] \ar[r]& \bullet^{v_2}\ar@/^.5pc/[d] \ar[d]\\
		\bullet^{v_3} \ar@/^.5pc/[u]\ar[u] & \bullet^{v_4} \ar@/^.5pc/[l] \ar[l]}$& $\ C_4^2=$&$
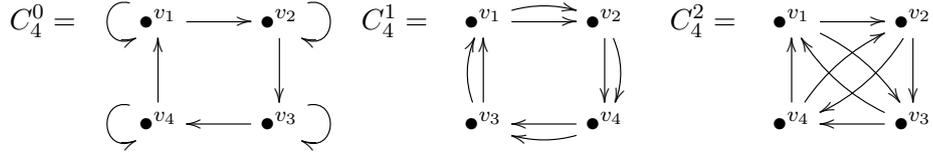
\xymatrix{\bullet^{v_1} \ar[r] \ar@/^.5pc/[rd]& \bullet^{v_2} \ar[d] \ar@/^.5pc/[dl]\\
		\bullet^{v_4} \ar[u] \ar@/^.5pc/[ur]& \bullet^{v_3} \ar[l] \ar@/^.5pc/[lu]}$\\
\end{tabular}
\bigskip
\medskip

%In \cite{ap:tlpaogcg} the authors computed the size of the Grothendieck of the Leavitt path algebra $L_K(C_n^j)$. 

We next provide a criterion for when the Leavitt path algebra of the Cayley graph $E_{G, S}$, associated with a finite group $G$ and a generating set $S$, has the gr-IBN property. This result follows from \cite[Theorem 4.2]{np:tsolpaatibnp}.

\begin{prop}\label{Caley}
	Let $G$ be an arbitrary finite group, $S$ its generating set. The following conditions are equivalent:
	
	(1) $L_K(E_{G, S})$ has IBN;
	
	(2) $L_K(E_{G, S})$ has gr-IBN;
	
	(3) $S$ contains only one element;
	
	(4) $E_{G, S}$ is a graph consisting of a single cycle with $|G|$ vertices.
\end{prop}
\begin{proof}
	(1) $\Longrightarrow$ (2). It follows from Corollary \ref{UGNIBNgrIBN}.
	
	(2) $\Longrightarrow$ (3). Assume that $L_K(E_{G, S})$ has gr-IBN, and $|S|\geq 2$. We immediately have that
	$|s^{-1}(g)|=|r^{-1}(g)|=|S|\geq 2$ for all $g\in G$. So, if $A=(a_{ij})_h$ is the adjacency matrix of $E_{G, S}$, then  $$\sum\limits_{i=1}^ha_{i1}=\sum\limits_{i=1}^ha_{i2}=\cdots =\sum\limits_{i=1}^ha_{ih}=|S|\geq 2.$$
	
	By Corollary \ref{collum}, $L_K(E_{G, S})$ has no gr-IBN, a contradiction. Thus, (2) $\Longrightarrow$ (3).

	(3) $\Longrightarrow$ (4). Assume that $S$ contains only one element. If $S = \{g\}\ (g\in G)$, then $G$ is a cyclic group generated by $g$, and hence, every vertex $v_{g^i}\ (0\leq i\leq |G|)$ emits only one edge to the vertex $v_{g^{i+1}}$, that means,
	the graph $E_{G, S}$ is of the form $$
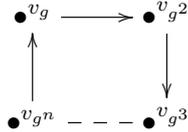
\xymatrix{\bullet^{v_g} \ar[r] & \bullet^{v_{g^2}} \ar[d] \\
		\bullet^{v_{g^n}} \ar[u] & \bullet^{v_{g^3}}\ar@{--}[l] }$$
	if $n = |G|\geq 2$ and, $E_{G, S}$ is a loop 
	in the case when $n = |G|=1$. 
	
	(4) $\Longrightarrow$ (1). It follows from Corollary \ref{maximal}, thus completing the~proof.
\end{proof}

\subsection{Hopf graphs}

In this subsection, we investigate a class of graphs related to Cayley graphs. Specifically, we study the gr-IBN property for Leavitt path algebras of finite Hopf graphs. We next recall the notion of Hopf graphs introduced by Cibils and Rosso in \cite{Hopf}. Let $G$ be an arbitrary group and $\mathcal{C}$ the set of all conjugacy classes of $G$. We call a \emph{ramification data} of $G$ is a function $\mathfrak{r}: \mathcal{C}\longrightarrow \mathbb{N}$, denoted by $\mathfrak{r}= \sum\limits_{C\in \mathcal{C}} \mathfrak{r}_CC$ (where $\mathfrak{r}_C := \mathfrak{r}(C)$).  

%The {\it support} of $\mathfrak{r}$ is the set

%$$\rm{supp}(\mathfrak{r})=\{C\in \mathcal{C}\mid \mathfrak{r}_C>0\}.$$ We say that $\mathfrak{r}$ has \emph{finite support} if $\bigcup_{C\in \rm{supp}(\mathfrak{r})}C$ is a finite set. We denote by $S_{G, \mathfrak{r}}$ the \emph{subsemigroup} of $G$ generated by $\bigcup_{C\in \rm{supp}(\mathfrak{r})}C$. 
%In the group $G$, $\tau(G)=\{g\in G| g^n=1_G \text{ for some } n\in \mathbb{N}\}$ be called \emph{the torsion of} $G$.

\begin{defn}[{\cite[Definition 3.1]{Hopf}}]\label{Hopf}
	Let $G$ be an arbitrary group with a ramification data $\mathfrak{r}= \sum\limits_{C\in \mathcal{C}} \mathfrak{r}_CC$. The {\it Hopf graph associated to the pair} $(G,\mathfrak{r})$, denoted by $\Gamma_{G, \mathfrak{r}}$, has set of vertices $\Gamma_{G, \mathfrak{r}}^0=G$  and has $\mathfrak{r}_C$ edges from $x$ to $xc$	for each $x \in  G$ and $c\in C$.
\end{defn}

%It is worth mentioning that in \cite[Theorem 3.3]{Hopf}  Cibils and Rosso showed that Hopf graphs are precisely the graphs such that the path algebra can be endowed with a graded Hopf algebra structure.\medskip

For clarification, we illustrate the notion of Hopf graphs by presenting the following examples.

%VIET MA TRAN KE CUA CAC DO THI NAY RA VA NHAN XET TONG CAC COT BANG NHAU

\begin{exas}\label{2.3}
	Let $G = S_3$ be the symmetric group of order $6$, and write $$G=\{id, (12), (13), (23), (123), (132)\}.$$ We then have $\mathcal{C}=\{[id], [(12)], [(123)]\}$, where $$[id]=\{id\}, [(12)]=\{(12), (13), (23)\}, [(123)]=\{(123), (132)\}.$$
	
	(1) Consider the ramification data $\mathfrak{r}=[(123)]$. We then have that $\Gamma_{G, \mathfrak{r}}$ is the following graph:
	\begin{center}
		$
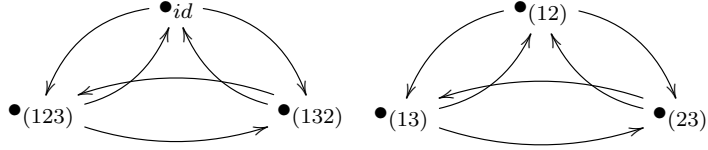
\xymatrix{& \bullet_{id} \ar@/^1pc/[rd] \ar@/_1pc/[ld]& \\
			\bullet_{(123)} \ar@/_1pc/[ru] \ar@/_1pc/[rr]& & \bullet_{(132)} \ar@/^1pc/[lu] \ar@/_1pc/[ll]}$   
		$\xymatrix{& \bullet_{(12)} \ar@/^1pc/[rd] \ar@/_1pc/[ld]& \\
			\bullet_{(13)} \ar@/_1pc/[ru] \ar@/_1pc/[rr]& & \bullet_{(23)} \ar@/^1pc/[lu] \ar@/_1pc/[ll]}$ 
	\end{center}
	\medskip
%	The adjacency matrix of $\Gamma_{G, \mathfrak{r}}$ is given by
%	$$
%	A_{\Gamma_{G, \mathfrak{r}}}=\begin{pmatrix}
%		0 & 1 & 1 & 0 & 0 & 0 \\
%		1 & 0 & 1 & 0 & 0 & 0 \\
%		1 & 1 & 0 & 0 & 0 & 0 \\
%		0 & 0 & 0 & 0 & 1 & 1 \\
%		0 & 0 & 0 & 1 & 0 & 1 \\
%		0 & 0 & 0 & 1 & 1 & 0 \\
%	\end{pmatrix}.
%	$$
	
	(2) Consider the ramification data $\mathfrak{r}=[(12)]$. Then, $\Gamma_{G, \mathfrak{r}}$ is the following graph:
	$$
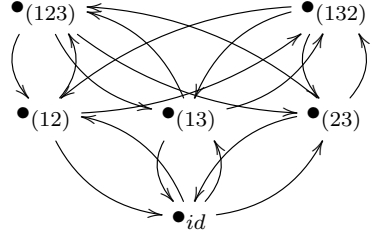
\xymatrix{\bullet_{(123)} \ar@/_1pc/[d] \ar@/_1pc/[rd] \ar@/_1pc/[rrd]& & \bullet_{(132)} \ar@/_1pc/[d] \ar@/_1pc/[ld] \ar@/_1pc/[lld]\\
		\bullet_{(12)} \ar@/_1pc/[u] \ar@/_1pc/[urr] \ar@/_1pc/[rd] & \bullet_{(13)} \ar@/_1pc/[d] \ar@/_1pc/[ul] \ar@/_1pc/[ru]& \bullet_{(23)} \ar@/_1pc/[u] \ar@/_1pc/[ull] \ar@/_1pc/[ld]\\
		& \bullet_{id} \ar@/_1pc/[u] \ar@/_1pc/[ul] \ar@/_1pc/[ru]& }$$
%		The adjacency matrix of $\Gamma_{G, \mathfrak{r}}$ is given by
%	$$
%	A_{\Gamma_{G, \mathfrak{r}}}=\begin{pmatrix}
%		0 & 1 & 0 & 1 & 0 & 1 \\
%		1 & 0 & 1 & 0 & 1 & 0 \\
%		0 & 1 & 0 & 1 & 0 & 1 \\
%		1 & 0 & 1 & 0 & 1 & 0 \\
%		0 & 1 & 0 & 1 & 0 & 1 \\
%		1 & 0 & 1 & 0 & 1 & 0 \\
%	\end{pmatrix}.
%	$$
\end{exas}

We note here that for a Hopf graph, the adjacency matrix has equal row sums and equal column sums.

	\begin{lem}\label{lem4.4}
		Let $G$ be a group with $h$ elements and $\mathfrak{r}=\sum_{C\in \mathcal{C}}\mathfrak{r}_CC$ a ramification data. If $A=(a_{ij})_h$ is the adjacency matrix of $\Gamma_{G, \mathfrak{r}}$, then  $$\sum\limits_{i=1}^ha_{i1}=\sum\limits_{i=1}^ha_{i2}=\cdots =\sum\limits_{i=1}^ha_{ih}=\sum_{C\in \mathcal{C}}\mathfrak{r}_C|C|, \text{ and }$$
		$$\sum\limits_{i=1}^ha_{1i}=\sum\limits_{i=1}^ha_{2i}=\cdots =\sum\limits_{i=1}^ha_{hi}=\sum_{C\in \mathcal{C}}\mathfrak{r}_C|C|.$$
	\end{lem}
	\begin{proof}
		By Definition \ref{Hopf}, for each $g \in G$, the edges leaving $g$ are in bijection with 
		$$\bigsqcup_{C\in \mathcal{C}} \{\mathfrak{r}_C \text{ edges to } gc \ |\ c \in C\},$$ giving
		$$
		|s^{-1}(g)|=\sum_{C\in \mathcal{C}}\mathfrak{r}_C|C|.
		$$
		Similarly, there are $\mathfrak{r}_C$ edges enters $g$ for each pair $(x,c)$ with $x\in G$, $c\in C$, $xc=g$;
		for each such $c$ there is a unique $x=g c^{-1}$, giving
		$$
		|r^{-1}(g)|=\sum_{C\in \mathcal{C}}\mathfrak{r}_C|C|.
		$$
		Hence all row sums and column sums of $A$ equal $\sum_{C\in \mathcal{C}}\mathfrak{r}_C|C|$.
	\end{proof}
	
	Combining Lemma \ref{lem4.4} and Corollary \ref{collum}, we obtain the following result, which shows that IBN and gr-IBN are equivalent for Leavitt path algebras of Hopf~graphs.
	\begin{prop}\label{HopfThm}

		Let $G$ be a finite group with  a ramification data $\mathfrak{r}=\sum_{C\in \mathcal{C}}\mathfrak{r}_CC$. Then, the following statements are equivalent:
		
		$(1)$ $L_K(\Gamma_{G, \mathfrak{r}})$ has IBN;
		
		$(2)$ $L_K(\Gamma_{G, \mathfrak{r}})$ has gr-IBN;
		
		$(3)$ $\sum_{C\in \mathcal{C}}\mathfrak{r}_C|C|\leq 1.$
	\end{prop}
\begin{proof}
	Let $|G|=h$ and $A=(a_{ij})_h$ be the adjacency matrix of $\Gamma_{G, \mathfrak{r}}$. Lemma \ref{lem4.4} shows that
	 $$\sum\limits_{i=1}^ha_{i1}=\sum\limits_{i=1}^ha_{i2}=\cdots =\sum\limits_{i=1}^ha_{ih}=\sum_{C\in \mathcal{C}}\mathfrak{r}_C|C|.$$
	Thus, by Corollary \ref{collum}, if $\sum_{C\in \mathcal{C}}\mathfrak{r}_C|C|\geq 2$ then $L(\Gamma_{G, \mathfrak{r}})$ has no gr-IBN. This implies that $(2) \Longrightarrow (3)$.
	
	$(3) \Longrightarrow (1)$. It follows from \cite[Corollary 3.10]{np:olpahg}.
%	If $k=0$, then  $\Gamma_{G, \mathfrak{r}}$ is a disjoint union of isolated vertices by  \cite[Corollary 2.11]{np:olpahg}. If $k=1$, by  \cite[Proposition 2.5]{np:olpahg}, $\Gamma_{G, \mathfrak{r}}$ has a cycle and by   \cite[Corollary 2.11]{np:olpahg}, $\Gamma_{G, \mathfrak{r}}$ is a disjoint union of single cycles. So, if $k\leq 1$, then $\Gamma_{G, \mathfrak{r}}$ has a maximal sink or cycle with finitely many predecessors. So, By Corollary \ref{maximal}, $L(\Gamma_{G, \mathfrak{r}})$ has IBN. 
	
	By Corollary \ref{UGNIBNgrIBN}, we obtain that (1) $\Longrightarrow$ (2). Thus finishing the proof.
\end{proof}

\subsection{Quotient graphs}

In \cite{Li2024}, Li, Li, and Wen showed that the study of the IBN property for the Leavitt path algebra of the graph $E$ can be reduced to the quotient graph $E/H$ where $H$ is a hereditary saturated subset of $E^0$. In this subsection, we demonstrate that an analogous result holds for the gr-IBN property.  
%Subsequently, we apply our results to the Cartesian product of graphs.

\begin{defn}[{\cite[Definition 2.4.11]{aam:lpa}}]\label{quotientgraph}
	Let $E$ be a finite graph, and let $H$ be a hereditary saturated subset of $E^0$.
	We denote by $E/H$ the following graph, which is called the {\it quotient graph} of $E$ by~$H$:
	$$
	(E/H)^0 = E^0 \setminus H, 
	\qquad
	(E/H)^1 = \{\, e \in E^1 : r(e) \notin H \,\}.
	$$
	The range and source functions for $E/H$ are defined by restricting the range and source functions of $E$ to $(E/H)^1$.
\end{defn}

We establish the graded analogue of \cite[Corollary 3.6]{Li2024}, which implies that the gr-IBN property
of $L(E/H)$ of the quotient graph $E/H$ implies the gr-IBN property of $L_K(E)$.

\begin{prop}\label{cor4.7}
	Let $E$ be a finite graph and $H$ a hereditary saturated subset of $E^0$. If $L_K(E/H)$ has gr-IBN, then $L_K(E)$ has gr-IBN.
\end{prop}
\begin{proof}
	According to \cite[Theorem 2.4.12]{aam:lpa}, there is a $\mathbb{Z}$-graded homomorphism:
	$$
	\Psi: L_K(E) \longrightarrow L_K(E/H)
	$$
	by setting:
	
		$
		\Psi(v) =
		\begin{cases}
			v, & \text{if } v \notin H \\
			0, & \text{otherwise}
		\end{cases}
		$\ \ \ 
		$
		\Psi(e) =
		\begin{cases}
			e, & \text{if } r(e) \notin H \\
			0, & \text{otherwise}
		\end{cases}
		$\ \ \ \ 
		$
		\Psi(e^*) =
		\begin{cases}
			e^*, & \text{if } s(e^*) \notin H \\
			0, & \text{otherwise}.
		\end{cases}
		$
	Suppose $L_K(E/H)$ has gr-IBN. From Lemma \ref{r--s}, it immediately deduce that $L_K(E)$ has gr-IBN. 
\end{proof}

	The following example shows that, in general, the gr-IBN property of $L_K(E/H)$ cannot be derived from that of  $L_K(E)$.
	\begin{exas}
	Let $E$ be the graph
	\medskip
	$$
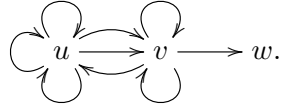
\xymatrix{u \ar@(ld,lu) \ar@(ul,ur) \ar@(dr,dl) \ar[r] \ar@/^.7pc/[r]& v \ar@(ul,ur) \ar@(dr,dl) \ar@/^.7pc/[l] \ar[r]& w.}$$
	\medskip
	
	Consider the hereditary saturated subset $H=\{w\}$ of $E^0$. Then, $E/H$ is the following graph:
	\medskip
	$$
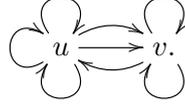
\xymatrix{u \ar@(ld,lu) \ar@(ul,ur) \ar@(dr,dl) \ar[r] \ar@/^.7pc/[r]& v. \ar@(ul,ur) \ar@(dr,dl) \ar@/^.7pc/[l]}$$
	\medskip
	
	According to Proposition \ref{hassink}, $L_K(E)$ has gr-IBN since $E$ has a sink $w$. On the other hand, by Example \ref{Ex:5.11}, $L_K(E/H)$ has no gr-IBN. \hfill $\Box$
\end{exas}

\subsection{Cartesian product of graphs}
In the remainder of this section, we apply the above results to the Leavitt path algebras of Cartesian product of graphs.
We next recall the definition of the Cartesian product of graphs  introduced by Feigenbaum \cite{Fei1986}.
\begin{defn}[{\cite[page 105]{Fei1986}}]\label{Cartesian}
Let $E$ and $F$ be two finite graphs. The Cartesian product of $E$ and $F$ is the following graph $G$:
$$
G^0 = \{(u,v) : u \in E^0,\ v \in F^0\},
$$
$$
G^1 = \{(e,v) : e \in E^1,\ v \in F^0\}
\;\cup\;
\{(u,f) : u \in E^0,\ f \in F^1\},
$$
with
$$
s_G(e,v) = (s_E(e),v), 
\qquad
s_G(u,f) = (u,s_F(f)),
$$
$$
r_G(e,v) = (r_E(e),v), 
\qquad
r_G(u,f) = (u,r_F(f)).
$$

Further, by $E \times F$ we denote the Cartesian product of $E$ and $F$.	
\end{defn}
\begin{exas}
	Let $E$ and $F$ be as follows:
	$$
	\xymatrix{E=& u_1\ar@/^.5pc/[r]^{e_1} & u_2 \ar@/^.5pc/[l]^{e_2}& & F= & v_1 \ar[r]^{f_1}& v_2 \ar[r]^{f_2}& v_3 }
	$$
	
	Then, the graph $E \times F$ is as follows:
	$$
	\xymatrix{(u_1,v_1) \ar[rr]^{(u_1,f_1)} \ar@/^.5pc/[d]^{(e_1,v_1)}&& (u_1,v_2)\ar[rr]^{(u_1,f_2)} \ar@/^.5pc/[d]^{(e_1,v_2)}&& (u_1,v_3) \ar@/^.5pc/[d]^{(e_1,v_3)}\\
	(u_2,v_1) \ar[rr]^{(u_2,f_1)} \ar@/^.5pc/[u]^{(e_2,v_1)}&& (u_2,v_2)\ar[rr]^{(u_2,f_2)} \ar@/^.5pc/[u]^{(e_2,v_2)}&& (u_2,v_3). \ar@/^.5pc/[u]^{(e_2,v_3)}}
	$$
\end{exas}
\begin{lem}\label{twosink}
	Let $E$ and $F$ be the finite graphs. Then, 
	$E \times F$ has a sink if and only if $E$ and $F$ have a sink.
\end{lem}
\begin{proof}
	By Definition \ref{Cartesian}, for any $(u,v)\in (E\times F)^0$, 
	$$
	s^{-1}_{E\times F}((u,v))= \{(e,v) : e \in E^1,\ s_E(e)=u\}
	\;\cup\;
	\{(u,f) : f \in F^1,\ s_F(f)=v\}.
	$$
	Hence, $s^{-1}_{E\times F}((u,v))=\emptyset$ if and only if $s^{-1}_E(u)=\emptyset$ and $s^{-1}_F(v)=\emptyset$. Thus finishing the proof.
\end{proof}
From Lemma \ref{twosink}, we can construct some graphs whose Leavitt path algebras have gr-IBN.
\begin{prop}
	 Let $E$ and $F$ be the finite graphs. If both $E$ and $F$ have a sink, then $L_K(E \times F)$ has gr-IBN.
\end{prop}
\begin{proof}
	It immediately follows from Proposition \ref{hassink} and Lemma \ref{twosink}.
\end{proof}
The following statement provides a sufficient condition for Leavitt path algebra associated to a Cartesian product of graphs has (gr-)IBN. This is an extension of \cite[Proposition 3.15]{Li2024}.
\begin{prop}\label{source}
	Let $E$ be a finite graph without sinks and $F$ a nontrivial finite graph with a source. If $L_K(E)$ has $($gr-$)$IBN, then $L_K(E \times F)$ has $($gr-$)$IBN.
\end{prop}
\begin{proof}
	Let $F^0=\{v_1, v_2, \cdots, v_n\}$, where $v_1$ is a source.	Set
	$$
	H = \{(u,v_i) : u \in E^0,\ i = 2,3,\ldots,n\}.
	$$
	
	We prove that $H$ is a nontrivial hereditary saturated subset of $(E \times F)^0$. Indeed, for any $(u,v_i) \in H$, 
	$$
	s^{-1}_{E\times F}((u,v_i))= \{(e,v_i) : e \in E^1,\ s_E(e)=u\}
	\;\cup\;
	\{(u,f) : f \in F^1,\ s_F(f)=v_i\}.
	$$
	Hence, 
	$$
	r_{E\times F}(s^{-1}_{E\times F}((u,v_i)))= \{(u',v_i) : u'\in r_E(s^{-1}_E(u))\}
	\;\cup\;
	\{(u,v) : v\in r_F(s^{-1}_F(v_i))\}.
	$$
	By definition of $H$, $\{(u',v_i) : u'\in r_E(s^{-1}_E(u))\}\subseteq H$. Since $v_1$ is a source in $F$, we get that $v_1\notin r_F(s^{-1}_F(v_i))$. So, $\{(u,v) : v\in r_F(s^{-1}_F(v_i))\}\subseteq H$. Thus, $H$ is hereditary.
	
	Next, we show that $H$ is saturated. Note that since $E$ has no sinks, it follows that every $(u, v_i) \in (E \times F)^0$ is a regular vertex. For every $(u, v_i) \in (E \times F)^0$, if
	$$
	r(s^{-1}(u, v_i)) \subseteq H,
	$$
	we claim that $i \ge 2$, i.e., $(u,v_i)\in H$. Assume that $i = 1$; then we have
	$$
	r_{E\times F}(s^{-1}_{E\times F}((u,v_1)))= \{(u',v_1) : u'\in r_E(s^{-1}_E(u))\}
	\;\cup\;
	\{(u,v) : v\in r_F(s^{-1}_F(v_1))\}.
	$$
	This is a contradiction, because $(u', v_1) \notin H$ for any $u' \in E^0$. It follows that $H$ is saturated.

	It is straightforward to verify that the quotient graph $(E \times F)/H$ is exactly the given graph $E$.
	By \cite[Corollary 3.6]{Li2024} and Proposition \ref{cor4.7}, if $L_K(E)$ has (gr-)IBN, then $L_K(E \times F)$ has (gr-)IBN. 
	The proof is complete.
\end{proof}
We present an example to show that the \lq\lq has a source\rq\rq\  assumption in Proposition \ref{source} is necessary.

\begin{exas}
	Let $E$ and $F$ be as follows:
	$$
	\xymatrix{E=& u_1\ar@/^.5pc/[r]^{e_1} & u_2 \ar@/^.5pc/[l]^{e_2}& & F=& v_1\ar@/^.5pc/[r]^{f_1} & v_2 \ar@/^.5pc/[l]^{f_2}}
	$$
	
	Then, the graph $E \times F$ is as follows:
	$$
	\xymatrix{(u_1,v_1) \ar@/^.5pc/[rr]^{(u_1,f_1)} \ar@/^.5pc/[d]^{(e_1,v_1)}&& (u_1,v_2) \ar@/^.5pc/[ll]^{(u_1,f_2)} \ar@/^.5pc/[d]^{(e_1,v_2)}\\
		(u_2,v_1) \ar@/^.5pc/[rr]^{(u_2,f_1)} \ar@/^.5pc/[u]^{(e_2,v_1)}&& (u_2,v_2).\ar@/^.5pc/[ll]^{(u_2,f_2)}\ar@/^.5pc/[u]^{(e_2,v_2)}}
	$$

By Corollary \ref{maximal}, $L_K(E)$ has (gr-)IBN. But, apply Corollary \ref{collum}, we get that $L_K(E\times F)$ has no (gr-)IBN. \hfill $\Box$
\end{exas}

From Proposition \ref{source}, we obtain the following results, first established in \cite{Li2024} for IBN property.
Recall the following graph is the {\it n-line graph}:
\begin{equation}\label{n-line}
	\xymatrix{L_n=&v_1\ar[r]^{f_1}&v_2\ar[r]^{f_2}&\cdots \ar[r]^{f_{n-1}}&v_n,}
\end{equation}
and the following graph is the {\it m-cycle}:
\begin{equation}\label{m-cycle}
	\xymatrix{& & u_2\ar@/^.5pc/[r]&u_3 \ar@/^.5pc/[rd]&\\
		C_m=& u_1 \ar@/^.5pc/[ru]& & & u_4. \ar@/^.5pc/[dl]\\
		&& u_m \ar@/^.5pc/[lu]& u_5\ar@{..}@/^.5pc/[l]}
\end{equation}

\begin{cor}[{cf.~\cite[Proposition 3.15]{Li2024}}]\label{prop3.5}
	Let $E$ be a finite graph without sinks, and let $L_n$ be the $n$-line graph given in $(\ref{n-line})$. 
	If $L_K(E)$ has $($gr-$)$IBN, then $L_K(E \times L_n)$ has $($gr-$)$IBN.
\end{cor}
\begin{proof}
	Since $L_n$ has a source $v_1$,  $L(E \times L_n)$ has (gr-)IBN by Proposition \ref{source}.
\end{proof}

\begin{cor}[{cf.~\cite[Proposition 3.14]{Li2024}}]
	Let $C_m$ be the $m$-cycle given in $(\ref{m-cycle})$, and let $L_n$ be the $n$-line given in $(\ref{n-line})$. Then, $L_K(C_m \times L_n)$ has $($gr-$)$IBN.
\end{cor}
\begin{proof}
	By Corollary \ref{maximal}, $L_K(C_m)$ has (gr-)IBN. Then, by applying Corollary \ref{prop3.5}, we obtain that $L_K(C_m \times L_n)$ has (gr-)IBN. Since $C_m \times L_n$ is the graph in which every vertex is in at most one cycle, we can also apply Corollary \ref{atmost} to give an independent verification of this fact.
\end{proof}

	\vskip 0.5 cm \vskip 0.5cm {
		
	\end{document}